\documentclass[12pt]{article}
\usepackage[centertags]{amsmath}
\usepackage{amsfonts}
\usepackage{amssymb}
\usepackage{latexsym}
\usepackage{amsthm}
\usepackage{newlfont}
\usepackage{graphicx}
\usepackage{listings}
\usepackage{booktabs}
\usepackage{abstract}
\newlength{\defbaselineskip}
\newcommand{\setlinespacing}[1]%
           {\setlength{\baselineskip}{#1 \defbaselineskip}}

\newcommand{\N}{{\mathbb{N}}}

\newcommand{\actaqed}{\hfill $\actabox$}
{\medskip\noindent \textit{Proof of #1. }}%
{\actaqed \medskip}

\def \bx{{\mathbf x}}
\def \bk{{\mathbf k}}
\def \bn{{\mathbf n}}
\def \bs{{\mathbf s}}

\def \br{{\mathbf r}}

\def\mb{{\bar m}}

\def\Di{{\mathcal D}}
\def\D{{\mathcal D}}

\def\K{{\mathcal K}}

\def\Tr{{\mathcal T}}
\def\R{{\mathbb R}}
\def\T{{\mathbb T}}

\def \<{\langle}
\def\>{\rangle}
\def \La{\Lambda}
\def \ep{\epsilon}
\def \e{\epsilon}

\def \ff{\varphi}

\def\N{{\mathbb N}}
\def\bN{{\mathbf N}}
\def\Z{{\mathbb Z}}

\def \sp{\operatorname{span}}
\def \dim{\operatorname{dim}}

\def\la{\lambda}

\def\bN{\mathbf N}
\def\bW{\mathbf W}
\def\bH{\mathbf H}
\def\bB{\mathbf B}

\newcommand{\be}{\begin{equation}}
\newcommand{\ee}{\end{equation}}

\newtheorem{Theorem}{Theorem}[section]
\newtheorem{Lemma}{Lemma}[section]
\newtheorem{Definition}{Definition}[section]
\newtheorem{Proposition}{Proposition}[section]
\newtheorem{Remark}{Remark}[section]

\numberwithin{equation}{section}

\begin{document}
\title{{Constructive sparse trigonometric approximation and other problems for functions with mixed smoothness}\thanks{\it Math Subject Classifications.
primary:  41A65; secondary: 42A10, 46B20.}}
\author{  V. Temlyakov \thanks{ University of South Carolina, USA, and Steklov Institute of Mathematics, Russia. Research was supported by NSF grant DMS-1160841 }} \maketitle
\begin{abstract}
{Our main interest in this paper is to  study some approximation problems for classes of functions with mixed smoothness.   We use technique, based on a combination of results from hyperbolic cross approximation, which were obtained in 1980s -- 1990s, and  recent results on greedy approximation to obtain sharp estimates for best $m$-term approximation with respect to the trigonometric system. 
We give some observations on numerical integration and approximate recovery of functions with mixed smoothness. We prove lower bounds, which show that one cannot improve accuracy of sparse grids methods with $\asymp 2^nn^{d-1}$ points in the grid by adding $2^n$ arbitrary points. In case of numerical integration these lower bounds provide best known lower bounds for optimal cubature formulas and for sparse grids based cubature formulas. }
\end{abstract}

\section{Introduction}

Sparse approximation with respect to dictionaries is a very important topic in the high-dimensional approximation. The main motivation for the study of sparse approximation is that many real world signals can be well approximated by sparse ones.   Sparse approximation
automatically implies a need for nonlinear approximation, in particular, for greedy approximation. We give a brief description of a sparse approximation problem   and present a  discussion of the obtained results and their
relation to previous work. 
In a general setting we are working in a Banach space $X$ with a redundant system of elements $\D$ (dictionary $\D$). There is a solid justification of importance of a Banach space setting in numerical analysis in general and in sparse approximation in particular (see, for instance, \cite{Tbook}, Preface). Let $X$ be a real Banach space with norm $\|\cdot\|:=\|\cdot\|_X$. We say that a set of elements (functions) $\D$ from $X$ is a dictionary  if each $g\in \D$ has norm   one ($\|g\|=1$), and the closure of $\sp \D$ is $X$. A symmetrized dictionary is $\D^\pm:=\{\pm g:g\in \D\}$. 
For a nonzero element $g\in X$ we let $F_g$ denote a norming (peak) functional for $g$:
$$
\|F_g\|_{X^*} =1,\qquad F_g(g) =\|g\|_X.
$$
The existence of such a functional is guaranteed by the Hahn-Banach theorem.

An element (function, signal) $s\in X$ is said to be $m$-sparse with respect to $\D$ if
it has a representation $s=\sum_{i=1}^mc_ig_i$,   $g_i\in \D$, $i=1,\dots,m$. The set of all $m$-sparse elements is denoted by $\Sigma_m(\D)$. For a given element $f$ we introduce the error of best $m$-term approximation
$$
\sigma_m(f,\D)_X := \inf_{s\in\Sigma_m(\D)} \|f-s\|_X.
$$
For a function class $\bW$ define
$$
\sigma_m(\bW,\D)_X :=\sup_{f\in\bW} \sigma_m(f,\D)_X.
$$
 Let 
$t\in (0,1]$ be a given   nonnegative number.  We define (see \cite{T8}) the Weak Chebyshev Greedy Algorithm (WCGA) that is a generalization for Banach spaces of Weak Orthogonal Greedy Algorithm defined and studied in \cite{T13} (see also \cite{Tbook}).

 {\bf Weak Chebyshev Greedy Algorithm (WCGA).}
We define $f_0 := f^{c,t}_0 :=f$. Then for each $m\ge 1$ we inductively define

1). $\varphi_m :=\varphi^{c,t}_m \in {\mathcal D}$ is any element satisfying
$$
|F_{f_{m-1}}(\varphi_m)| \ge t\sup_{g\in {\mathcal D}} |F_{f_{m-1}}(g)|.
$$

2). Define
$$
\Phi_m := \Phi^t_m := \text{span} \{\varphi_j\}_{j=1}^m,
$$
and define $G_m := G_m^{c,t}$ to be the best approximant to $f$ from $\Phi_m$.

3). Denote
$$
f_m := f^{c,t}_m := f-G_m.
$$

We demonstrated in the paper \cite{T144} that the Weak Chebyshev Greedy Algorithm (WCGA)  is very good for $m$-term approximation with respect to a special class of dictionaries, in particular, for the trigonometric system.   The trigonometric system is a classical system that is known to be difficult to study. In \cite{T144} we study among other problems the problem of nonlinear sparse approximation with respect to it. Let  ${\mathcal R}{\mathcal T}$ denote the real trigonometric system 
$1,\sin 2\pi x,\cos 2\pi x, \dots$ on $[0,1]$ and let ${\mathcal R}{\mathcal T}_p$ to be its version normalized in $L_p([0,1])$. Denote ${\mathcal R}{\mathcal T}_p^d := {\mathcal R}{\mathcal T}_p\times\cdots\times {\mathcal R}{\mathcal T}_p$ the $d$-variate trigonometric system. We need to consider the real trigonometric system because the algorithm WCGA is well studied for the real Banach space.  
   We proved in \cite{T144}  the following 
Lebesgue-type inequality for the WCGA.
\begin{Theorem}\label{TC} Let $\D$ be the normalized in $L_p$, $2\le p<\infty$, real $d$-variate trigonometric
system. Then    
for any $f\in L_p$ the WCGA with weakness parameter $t$ gives
\begin{equation}\label{I1.4}
\|f_{C(t,p,d)m\ln (m+1)}\|_p \le C\sigma_m(f,\D)_p .
\end{equation}
\end{Theorem}

The above Lebesgue-type inequality guarantees that the WCGA works very well for each individual function $f$. As a complement to this inequality we would like to obtain results, which relate rate of decay of $\sigma_m(f,\Tr^d)_p$ to some smoothness type properties of 
$f$. It is the main goal of this paper. We measure smoothness in terms of mixed derivative and mixed difference. We note that the function classes with bounded mixed derivative are not only interesting and challenging object for approximation theory but they are important in numerical computations. Griebel and his group 
use approximation methods designed for these classes in elliptic variational
problems.
   Recent work of Yserentant  on new
regularity models for the Schr\"odinger equation shows that the
eigenfunctions of the electronic Schr\"odinger operator have a certain mixed
smoothness similar to the bounded mixed derivative. This makes approximation techniques developed for classes of functions with bounded mixed derivative a proper choice for the numerical
treatment of the Schr\"odinger
equation.      

Sparse trigonometric approximation of periodic functions began by the paper of S.B. Stechkin \cite{St}, who used it in the criterion for absolute convergence of trigonometric series. R.S. Ismagilov \cite{I} found nontrivial estimates for $m$-term approximation of functions with singularities of the type $|x|$ and gave interesting and important applications to the widths of Sobolev classes. He used a deterministic method based on number theoretical constructions. 
His method was developed by V.E. Maiorov \cite{M}, who used a method based on Gaussian sums.  Further strong results were obtained in \cite{DT1} with the help of a nonconstructive result from finite dimensional Banach spaces due to E.D. Gluskin \cite{G}. Other powerful nonconstructive method, which is based on a probabilistic argument, was used by Y. Makovoz \cite{Mk} and by E.S. Belinskii \cite{Be1}. Different methods were created in \cite{T29}, \cite{KTE1}, \cite{T1}, \cite{Tappr} for proving lower bounds for function classes.
It was discovered in \cite{DKTe} and \cite{T12} that greedy algorithms can be used for constructive $m$-term approximation with respect to the trigonometric system. We demonstrate in this paper how greedy algorithms can be used to prove optimal or best known upper bounds for $m$-term approximation of classes of function with mixed smoothness.  
It is a simple and powerful method of proving upper bounds. However, we do not know how to use it for small smoothness. The reader can find a detailed study of $m$-term approximation of classes of function with mixed smoothness, including small smoothness, in 
the paper \cite{Rom1} by A.S. Romanyuk. We note that in the case $2<p<\infty$ the upper bounds in \cite{Rom1} are not constructive.

We begin with some notation.
Let $\mathbf s=(s_1,\dots,s_d )$ be a  vector  whose  coordinates  are
nonnegative integers
$$
\rho(\mathbf s) := \bigl\{ \mathbf k\in\mathbb Z^d:[ 2^{s_j-1}] \le
|k_j| < 2^{s_j},\qquad j=1,\dots,d \bigr\},
$$
$$
Q_n :=   \cup_{\|\mathbf s\|_1\le n}
\rho(\mathbf s) \quad\text{--}\quad\text{a step hyperbolic cross},
$$
$$
\Gamma(N) := \bigl\{ \mathbf k\in\mathbb Z^d :\prod_{j=1}^d
\max\bigl( |k_j|,1\bigr) \le N\bigr\}\quad\text{--}\quad\text{a hyperbolic cross}.
$$
 For $f\in L_1 (\T^d)$
$$
\delta_{\mathbf s} (f,\mathbf x) :=\sum_{\mathbf k\in\rho(\mathbf s)}
\hat f(\mathbf k)e^{i(\mathbf k,\mathbf x)}.
$$
Let $G$ be a finite set of points in $\mathbb Z^d$, we denote
$$
\Tr(G) :=\left\{ t : t(\mathbf x) =\sum_{\mathbf k\in G}c_{\mathbf k}
e^{i(\mathbf k,\mathbf x)}\right\} .
$$
For the sake of simplicity we shall write  
$\Tr\bigl(\Gamma(N)\bigr) = \Tr(N)$. 

We study some approximation problems for classes of functions with mixed smoothness. We define these classes momentarily.  
We will begin with the case of univariate periodic functions. Let for $r>0$ 
\be\label{6.3}
F_r(x):= 1+2\sum_{k=1}^\infty k^{-r}\cos (kx-r\pi/2) 
\ee
and
\be\label{6.4}
W^r_p := \{f:f=\varphi \ast F_r,\quad \|\varphi\|_p \le 1\}.  
\ee
It is well known that for $r>1/p$ the class $W^r_p$ is embedded into the space of continuous functions $C(\T)$. In a particular case of $W^1_1$ we also have embedding into $C(\T)$.

In the multivariate case for $\bx=(x_1,\dots,x_d)$ denote
$$
F_r(\bx) := \prod_{j=1}^d F_r(x_j)
$$
and
$$
\bW^r_p := \{f:f=\varphi\ast F_r,\quad \|\varphi\|_p \le 1\}.
$$
For $f\in \bW^r_p$ we will denote $f^{(r)} :=\varphi$ where $\varphi$ is such that $f=\varphi\ast F_r$.

The main results of Section 2 are the following two theorems. We use the notation 
$\beta:=\beta(q,p):= 1/q-1/p$ and $\eta:=\eta(q):= 1/q-1/2$. In the case of trigonometric system $\Tr^d$ we drop it from the notation:
$$
\sigma_m(\bW)_p := \sigma_m(\bW,\Tr^d)_p.
$$
\begin{Theorem}\label{T2.8I} We have
$$
 \sigma_m(\bW^r_q)_{p}
  \asymp  \left\{\begin{array}{ll} m^{-r+\beta}(\log m)^{(d-1)(r-2\beta)}, & 1<q\le p\le 2,\quad r>2\beta,\\
 m^{-r+\eta}(\log m)^{(d-1)(r-2\eta)}, & 1<q\le 2\le p<\infty,\quad r>1/q,\\ 
 m^{-r}(\log m)^{r(d-1)}, & 2\le q\le p<\infty, \quad r>1/2.\end{array} \right.
$$
\end{Theorem}
\begin{Theorem}\label{T2.9I} We have
$$
 \sigma_m(\bW^r_q)_{\infty}
  \ll  \left\{\begin{array}{ll}   m^{-r+\eta}(\log m)^{(d-1)(r-2\eta)+1/2}, & 1<q\le 2,\quad r>1/q,\\ 
 m^{-r}(\log m)^{r(d-1)+1/2}, & 2\le q<\infty, \quad r>1/2.\end{array} \right.
$$
\end{Theorem}
The case $1<q\le p\le 2$ in Theorem \ref{T2.8I}, which corresponds to the first line, was proved in \cite{T29} (see also \cite{Tmon}, Ch.4). The proofs from \cite{T29} and \cite{Tmon} are constructive. In Section 2 we concentrate on the case $p\ge 2$. We use recently developed techniques on greedy approximation in Banach spaces to prove Theorems \ref{T2.8I} and \ref{T2.9I}. It is important that greedy approximation allows us not only to prove the above theorems but also to provide a constructive way for building the corresponding $m$-term approximants. We give a precise formulation.
\begin{Theorem}\label{T2.8C} For $p\in(1,\infty)$ and $\mu>0$ there exist constructive methods $A_m(f,p,\mu)$, which provide for $f\in \bW^r_q$ an $m$-term approximation such that
$$
 \|f-A_m(f,p,\mu)\|_p
 $$
 $$
  \ll  \left\{\begin{array}{ll} m^{-r+\beta}(\log m)^{(d-1)(r-2\beta)}, & 1<q\le p\le 2,\quad r>2\beta+\mu,\\
 m^{-r+\eta}(\log m)^{(d-1)(r-2\eta)}, & 1<q\le 2\le p<\infty,\quad r>1/q+\mu,\\ 
 m^{-r}(\log m)^{r(d-1)}, & 2\le q\le p<\infty, \quad r>1/2+\mu.\end{array} \right.
$$
\end{Theorem}
Similar modification of Theorem \ref{T2.9I} holds for $p=\infty$.
We do not have matching lower bounds for the upper bounds in Theorem \ref{T2.9I} in the case of approximation in the uniform norm $L_\infty$. 
In Section 3 we use known results on the entropy numbers to prove one lower bound in the case of functions of two variables. We note that it is of interest for small smoothness: $r<1/2$. 

As a direct corollary of Theorems \ref{TC} and \ref{T2.8I} we obtain the following result.
\begin{Theorem}\label{TCW} Let $p\in [2,\infty)$. Apply the WCGA with weakness parameter $t\in (0,1]$ to $f\in L_p$ with respect to the real trigonometric system ${\mathcal R}{\mathcal T}_p^d$. If $f\in \bW^r_q$, then we have
$$
 \|f_m\|_{p}
  \ll  \left\{\begin{array}{ll}  
 m^{-r+\eta}(\log m)^{(d-1)(r-2\eta)+r-\eta}, & 1<q\le 2,\quad r>1/q,\\ 
 m^{-r}(\log m)^{rd}, & 2\le q<\infty, \quad r>1/2.\end{array} \right.
$$
\end{Theorem}

In Sections 4 and 5 we give some observations on numerical integration and approximate recovery of functions with mixed smoothness. We prove there some lower bounds, which show that we cannot improve accuracy of sparse grids methods with $\asymp 2^nn^{d-1}$ points in the grid by adding $2^n$ arbitrary points. In case of numerical integration these lower bounds provide best known lower bounds for optimal cubature formulas and for sparse grids based cubature formulas. 

Our technique is based on a combination of results from the hyperbolic cross approximation, which were obtained in 1980s -- 1990s, and  recent results on greedy approximation. 
We formulate some known results from the hyperbolic cross approximation theory, which will be used in our analysis. 
We begin with the problem of estimating $\|f\|_p$ in terms of
the array $\bigl\{ \|\delta_{\bs} (f)\|_q  \bigr\}$.
Here and below $p$ and $q$  are  scalars  such
that $1\le q,p\le \infty$. Let an array
$\varepsilon = \{\varepsilon_{\bs}\}$ be given, where
$\varepsilon_{\bs}\ge 0$, $\bs = (s_1 ,\dots,s_d)$,
and $s_j$ are nonnegative integers, $j = 1,\dots,d$.
We denote by $G(\varepsilon,q)$ and $F(\varepsilon,q)$
the following sets of functions
$(1\le q\le \infty)$:
$$
G(\varepsilon,q) := \bigl\{ f\in L_q  : \bigl\|\delta_{\bs} (f)
\bigr\|_q
\le\varepsilon_{\bs}\qquad\text{ for all }\bs\bigr\} ,
$$
$$
F(\varepsilon,q) := \bigl\{ f\in L_q  : \bigl\|\delta_{\bs} (f)
\bigr\|_q
\ge\varepsilon_{\bs}\qquad\text{ for all }\bs\bigr\}.
$$

The following theorem is from \cite{Tmon}, p.29. For the special case $q=2$, which will be used in this paper, see \cite{T29} and \cite{Tmon}, p.86. 

\begin{Theorem}\label{T1.1} The following relations hold:
\begin{equation}\label{1.1}
\sup_{f\in G(\varepsilon,q)}\|f\|_p \asymp\left(\sum_{\bs}
\varepsilon_{\bs}^p2^{\|\bs\|_1(p/q-1)}\right)^{1/p},
\qquad 1\le q < p < \infty ;
\end{equation}
\begin{equation}\label{1.2}
\inf_{f\in F(\varepsilon,q)}\|f\|_p \asymp\left(\sum_{\bs}
\varepsilon_{\bs}^p  2^{\|\bs\|_1(p/q-1)}\right)^{1/p},
\qquad 1< p < q\le\infty ,
\end{equation}
with constants independent of $\varepsilon$.
\end{Theorem}
We will need a corollary of Theorem \ref{T1.1} (see \cite{Tmon}, Ch.1, Theorem 2.2), which we formulate as a theorem.
\begin{Theorem}\label{A} Let $1<q\le 2$. For any $t\in \Tr(N)$ we have
$$
\|t\|_A:=\sum_\bk |\hat t(\bk)| \le C(q,d) N^{1/q} (\log N)^{(d-1)(1-1/q)}\|t\|_q.
$$
\end{Theorem}
 Let
$$
\Pi(\bN,d) :=\bigl \{(a_1,\dots,a_d)\in \R^d  : |a_j|\le N_j,\
j = 1,\dots,d \bigr\} ,
$$
where $N_j$ are nonnegative integers and $\bN:=(N_1,\dots,N_d)$. We denote
$$
\Tr(\bN,d):=\{t:t = \sum_{\bk\in \Pi(\bN,d)} c_\bk e^{i(\bk,\bx)}\}.
$$
Then 
$$
\dim \Tr(\bN,d) = \prod_{j=1}^d (2N_j  + 1) =: \vartheta(\bN).
$$

The following theorem is from \cite{TBook}, Ch.2, Theorem 1.1 (see, also, \cite{TE2}). 
\begin{Theorem}\label{T1.2} Let $\varepsilon > 0$ and a subspace
$\Psi\subset \Tr(\bN,d)$ be such that
$\dim \Psi\ge\varepsilon\dim \Tr(\bN,d)$.
Then there is a $t \in\Psi$ such that
$$
\|t\|_{\infty}  = 1  ,\qquad \|t\|_2\ge C(\varepsilon,d) > 0 .
$$
\end{Theorem}

\section{Sparse approximation}

   For a Banach space $X$ we define the modulus of smoothness
$$
\rho(u) := \sup_{\|x\|=\|y\|=1}(\frac{1}{2}(\|x+uy\|+\|x-uy\|)-1).
$$
The uniformly smooth Banach space is the one with the property
$$
\lim_{u\to 0}\rho(u)/u =0.
$$

It is well known (see for instance \cite{DGDS}, Lemma B.1) that in the case $X=L_p$, 
$1\le p < \infty$, we have
\begin{equation}\label{t6.2}
\rho(u) \le \begin{cases} u^p/p & \text{if}\quad 1\le p\le 2 ,\\
(p-1)u^2/2 & \text{if}\quad 2\le p<\infty. \end{cases}     
\end{equation}
 
 Denote by $A_1({\mathcal D}):=A_1(\Di,X)$ the closure in $X$ of the convex hull of ${\mathcal D}$.
 The following theorem from \cite{T8} gives the rate of convergence of the WCGA for $f$ in $A_1({\mathcal D^\pm})$.
\begin{Theorem}\label{T2.1} Let $X$ be a uniformly smooth Banach space with the modulus of smoothness $\rho(u) \le \gamma u^q$, $1<q\le 2$. Then for $t\in (0,1]$ we have for any $f\in A_1({\mathcal D^\pm})$ that 
$$
\|f-G^{c,t}_m(f,{\mathcal D})\| \le C(q,\gamma)(1+ m t^p)^{-1/p},\quad p:= \frac{q}{q-1},
$$
with a constant $C(q,\gamma)$ which may depend only on $q$ and $\gamma$.
\end{Theorem}
\begin{Remark}\label{R2.1} It follows from the proof of Theorem \ref{T2.1} that $C(q,\gamma) \le C(q)\gamma^{1/q}$.
\end{Remark}

We proceed to  the Incremental Greedy Algorithm (see \cite{T12} and \cite{Tbook}, Chapter 6).       Let $\ep=\{\ep_n\}_{n=1}^\infty $, $\ep_n> 0$, $n=1,2,\dots$ . 

 {\bf Incremental Algorithm with schedule $\ep$ (IA($\ep$)).} 
  Denote $f_0^{i,\ep}:= f$ and $G_0^{i,\ep} :=0$. Then, for each $m\ge 1$ we have the following inductive definition.

(1) $\ff_m^{i,\ep} \in \Di$ is any element satisfying
$$
F_{f_{m-1}^{i,\ep}}(\ff_m^{i,\ep}-f) \ge -\ep_m.
$$

(2) Define
$$
G_m^{i,\ep}:= (1-1/m)G_{m-1}^{i,\ep} +\ff_m^{i,\ep}/m.
$$

(3) Let
$$
f_m^{i,\ep} := f- G_m^{i,\ep}.
$$

\begin{Theorem}\label{T2.2} Let $X$ be a uniformly smooth Banach space with  modulus of smoothness $\rho(u)\le \gamma u^q$, $1<q\le 2$. Define
$$
\ep_n := v\gamma ^{1/q}n^{-1/p},\qquad p=\frac{q}{q-1},\quad n=1,2,\dots .
$$
Then, for any $f\in A_1(\D)$,   we have
$$
\|f_m^{i,\ep}\| \le C(v) \gamma^{1/q}m^{-1/p},\qquad m=1,2\dots.
$$
\end{Theorem}

In \cite{T12} we demonstrated the power of the WCGA   in classical areas of harmonic analysis. The problem concerns the trigonometric $m$-term approximation in the uniform norm.  
The first result that indicated an advantage of $m$-term approximation with respect to the real trigonometric system ${\mathcal R}{\mathcal T}$  over approximation by trigonometric polynomials of order $m$ is due to Ismagilov \cite{I}
\begin{equation}\label{2.2}
\sigma_m(|\sin 2\pi x|,{\mathcal R}{\mathcal T} )_\infty \le C_\epsilon m^{-6/5+\epsilon},\quad\text{for any}\quad \epsilon >0.  
\end{equation}
Maiorov \cite{M} improved the estimate (\ref{2.2}):
\begin{equation}\label{2.3}
\sigma_m(|\sin 2\pi x|,{\mathcal R}{\mathcal T} )_\infty \asymp m^{-3/2}.  
\end{equation}

Both R.S. Ismagilov \cite{I} and V.E. Maiorov \cite{M} used constructive methods to get their estimates (\ref{2.2}) and (\ref{2.3}). V.E. Maiorov \cite{M} applied a number theoretical method based on Gaussian sums. 
 The key point of that technique can be formulated in terms of best $m$-term approximation of trigonometric polynomials. Let   ${\mathcal R}{\mathcal T}(N)$ be the subspace of real trigonometric polynomials of order $N$.
 Using the Gaussian sums one can prove (constructively) the estimate
\begin{equation}\label{2.4}
 \sigma_m(t,{\mathcal R}{\mathcal T})_\infty \le CN^{3/2}m^{-1}\|t\|_1,\quad t\in {\mathcal R}{\mathcal T}(N).  
\end{equation}
Denote  
$$
\|a_0+\sum_{k=1}^N(a_k\cos k2\pi x +b_k\sin k2\pi x)\|_A := |a_0| +\sum_{k=1}^N(|a_k|+|b_k|).
$$
We note that by simple inequality 
$$
\|t\|_A \le CN \|t\|_1 ,\quad t\in{\mathcal R}{ \mathcal T}(N),
$$
the estimate (\ref{2.4}) follows from the estimate
\begin{equation}\label{2.5}
\sigma_m(t,{\mathcal R}{\mathcal T})_\infty \le C(N^{1/2}/m)\|t\|_A,\quad t\in{\mathcal R}{ \mathcal T}(N).  
\end{equation}
Thus, (\ref{2.5}) is stronger than (\ref{2.4}). The following estimate was proved in \cite{DT1}
\begin{equation}\label{2.6}
\sigma_m(t,{\mathcal R}{\mathcal T})_\infty \le Cm^{-1/2}(\ln(1+ N/m))^{1/2}\|t\|_A,\quad t\in{\mathcal R}{ \mathcal T}(N).  
\end{equation}
In a way (\ref{2.6}) is much stronger than (\ref{2.5}) and (\ref{2.4}). The proof of (\ref{2.6}) from \cite{DT1} is not constructive. The estimate (\ref{2.6}) has been proved in \cite{DT1} with the help of a nonconstructive theorem of Gluskin \cite{G}. E.S. Belinskii \cite{Be1} used a probabilistic method to prove the following inequality: for $2\le p<\infty$
$$
\sigma_m(t,{\mathcal R}{\mathcal T})_\infty \le C(N/m)^{1/p}(\ln(1+ N/m))^{1/p}\|t\|_p,\quad t\in{\mathcal R}{ \mathcal T}(N).
$$
His proof is nonconstructive as well. 
In \cite{T12} we gave a constructive proof of (\ref{2.6}). The key ingredient of that proof is the WCGA. In the paper \cite{DKTe} we already pointed out that the WCGA provides a constructive proof of the estimate  \begin{equation}\label{2.7}
\sigma_m(f,{\mathcal R}{\mathcal T})_p \le C(p)m^{-1/2}\|f\|_A,\quad p\in [2,\infty).  
\end{equation}
The known proofs (before \cite{DKTe}) of (\ref{2.7}) were nonconstructive (see discussion in \cite{DKTe}, Section 5). Thus, the WCGA provides a way of building a good $m$-term approximant.  
We formulate here a result from \cite{T12}.  
 \begin{Theorem}\label{T2.3} There exists a constructive method $A(N,m)$ such that for any $t\in{\mathcal R}{ \mathcal T}(N)$ it provides an $m$-term trigonometric polynomial $A(N,m)(t)$ with the following approximation property
$$
\|t-A(N,m)(t)\|_\infty \le Cm^{-1/2}(\ln (1+N/m))^{1/2}\|t\|_A 
$$
with an absolute constant $C$.
\end{Theorem}
However, the step 2) of the WCGA makes it difficult to control the coefficients of the approximant -- they are obtained through the Chebyshev projection of $f$ onto $\Phi_m$. 
This motivates us to consider the IA($\ep$) which gives explicit coefficients of the approximant. 
An advantage of the IA($\ep$) over other greedy-type algorithms is that the IA($\ep$) gives precise control of the coefficients of the approximant. For all approximants $G_m^{i,\ep}$ we have the  property $\|G_m^{i,\ep}\|_A=1$. Moreover, we know that all nonzero coefficients of the approximant have the form $a/m$ where $a$ is a natural number. 
We prove the following result.
 \begin{Theorem}\label{T2.4} For any $t\in{\mathcal R}{ \mathcal T}(\bN,d)$ the IA($\ep$) applied to $f:=t/\|t\|_A$ provides after $m$ iterations an $m$-term trigonometric polynomial \newline$G_m(t):=G^{i,\ep}_m(f)\|t\|_A$ with the following approximation property
$$
\|t-G_m(t)\|_\infty \le C(d)(\mb)^{-1/2}(\ln \vartheta(\bN))^{1/2}\|t\|_A,
$$
$$
 \mb:=\max(1,m), \quad \|G_m(t)\|_A =\|t\|_A, 
$$
with a  constant $C(d)$, which may depend only on $d$.
\end{Theorem}
\begin{proof} It is clear that it is sufficient to prove Theorem \ref{T2.4} for $t\in{\mathcal R}{ \mathcal T}(\bN,d)$ with $\|t\|_A =1$. Then $t\in (A_1({\mathcal R}{ \mathcal T}(\bN,d)\cap{\mathcal R}{ \mathcal T}^d)^\pm,L_p)$ for all $p\in [2,\infty)$. Now, applying Theorem \ref{T2.2}   with 
$X=L_p$ and $\D^\pm$, where $\D:=\{ \ff_1,\ff_2,\dots,\ff_n\}$, $n=\vartheta(\bN)$, being the real trigonometric system 
$$
\ff_l:=\prod_{j\in E} \cos k_jx_j \prod_{j\in [1,d]\setminus E} \sin k_jx_j,
$$
we obtain that
\begin{equation}\label{2.8}
\|t -\sum_{j\in \La} \frac{a_j}{m}\ff_j\|_p \le C\gamma^{1/2}m^{-1/2},\qquad \sum_{j\in\La}|a_j| = m,
\end{equation}
where $\sum_{j\in \La} \frac{a_j}{m}\ff_j$ is the $G^{i,\ep}_m(t)$. 
By (\ref{t6.2}) we find $\gamma \le p/2$. Next, by the Nikol'skii inequality we get from (\ref{2.8})
$$
\|t -\sum_{j\in \La} \frac{a_j}{m}\ff_j\|_\infty \le C(d)N^{1/p}\| t -\sum_{j\in \La} \frac{a_j}{m}\ff_j\|_p \le C(d)p^{1/2}N^{1/p}m^{-1/2}.
$$
Choosing $p\asymp \ln N$ we obtain the desired in Theorem \ref{T2.4} bound.
\end{proof}

We point out that the above proof of Theorem \ref{T2.4} gives the following statement. 

 \begin{Theorem}\label{T2.5} Let $2\le p<\infty$. For any $t\in{\mathcal R}{ \mathcal T}(\bN,d)$ the IA($\ep$) applied to $f:=t/\|t\|_A$ provides after $m$ iterations an $m$-term trigonometric polynomial $G_m(t):=G^{i,\ep}_m(f)\|t\|_A$ with the following approximation property
$$
\|t-G_m(t)\|_p \le C(d)(\mb)^{-1/2}p^{1/2}\|t\|_A,\quad \mb:=\max(1,m), \quad \|G_m(t)\|_A =\|t\|_A, 
$$
with a  constant $C(d)$, which may depend only on $d$.
\end{Theorem}

We note that the implementation of the IA($\ep$) depends on the dictionary and the ambient space $X$. For example, for $d=1$ the IA($\ep$) from Theorem \ref{T2.4} acts with respect to the real trigonometric system \newline $1,\cos 2\pi x, \sin 2\pi x,\dots, \cos N2\pi x, \sin N2\pi x$ in the space $X=L_p$ with $p\asymp \ln N$. 

The above Theorems \ref{T2.4} and \ref{T2.5} are formulated for $m$-term approximation with respect to the real trigonometric system because the general Theorem \ref{T2.2} is proved for real Banach spaces. Clearly, as a corollary of Theorems \ref{T2.4} and \ref{T2.5} we obtain the corresponding results for the complex trigonometric system $\Tr^d :=\{e^{i(\bk,\bx)}\}_{\bk\in \Z^d}$. 

\begin{Theorem}\label{T2.5C} There exist constructive greedy-type approximation methods $G^p_m(\cdot)$, which provide $m$-term polynomials with respect to $\Tr^d$ with the following properties: for $2\le p<\infty$
\be\label{2.14p}
\|f-G^p_m(f)\|_p \le C_1(d)(\mb)^{-1/2}p^{1/2}\|f\|_A,\quad \|G^p_m(f)\|_A \le C_2(d)\|f\|_A,
\ee
  and for $p=\infty$, $f\in \Tr(\bN,d)$
\be\label{2.14'}
\|f-G^\infty_m(f)\|_\infty \le C_3(d)(\mb)^{-1/2}(\ln \vartheta(\bN))^{1/2}\|f\|_A,\quad \|G^\infty_m(f)\|_A \le C_4(d)\|f\|_A.
\ee
\end{Theorem}

We now apply Theorem \ref{T2.5C} for $m$-term approximation of functions with mixed smoothness. 
The following theorem was proved in \cite{T29} (see also \cite{Tmon}, Ch.4). The proofs from \cite{T29} and \cite{Tmon} are constructive.
We use the following notation 
$\beta:=\beta(q,p):= 1/q-1/p$ and $\eta:=\eta(q):= 1/q-1/2$.
\begin{Theorem}\label{T2.6} Let $1<q\le p\le 2$, $r>2\beta$. Then
$$
\sigma_m(\bW^r_q)_p \asymp m^{-r+\beta}(\log m)^{(d-1)(r-2\beta)}.
$$
\end{Theorem}
 First, we extend Theorem \ref{T2.6} to the case $1<q\le p <\infty$.

\begin{Theorem}\label{T2.8} One has
$$
 \sigma_m(\bW^r_q)_{p}
  \asymp  \left\{\begin{array}{ll} m^{-r+\beta}(\log m)^{(d-1)(r-2\beta)}, & 1<q\le p\le 2,\quad r>2\beta,\\
 m^{-r+\eta}(\log m)^{(d-1)(r-2\eta)}, & 1<q\le 2\le p<\infty,\quad r>1/q,\\ 
 m^{-r}(\log m)^{r(d-1)}, & 2\le q\le p<\infty, \quad r>1/2.\end{array} \right.
$$
\end{Theorem}
\begin{proof} The case $p\le 2$, which corresponds to the first line, follows from Theorem \ref{T2.6}. We note that in the case $p>2$ Theorem \ref{T2.8} is proved in \cite{T1}. However, the proof there is not constructive -- it uses a nonconstructive result from \cite{DT1}. We provide a constructive proof, which is based on greedy algorithms.  Also, this proof works under weaker conditions on $r$: $r>1/q$ instead of $r>1/q+\eta$ for $1<q\le2$. 
 The following  lemma plays the key role in the proof. 
\begin{Lemma}\label{L2.1}  Define for $f\in L_1$
$$
f_l:=\sum_{\|\bs\|_1=l}\delta_\bs(f), \quad l\in \N_0,\quad \N_0:=\N\cup \{0\}.
$$
Consider the class
$$
\bW^{a,b}_A:=\{f: \|f_l\|_A \le 2^{-al}l^{(d-1)b}\}.
$$
Then for $2\le p\le\infty$ and $0<\mu<a$ there is a constructive method $A_m(\cdot,p,\mu)$ based on greedy algorithms, which provides the bound for $f\in \bW^{a,b}_A$
\begin{equation}\label{2.13}
\|f-A_m(f,p,\mu)\|_p \ll  m^{-a-1/2} (\log m)^{(d-1)(a+b)},\quad 2\le p<\infty,      
\end{equation}
\begin{equation}\label{2.14}
\|f-A_m(f,p,\mu)\|_\infty \ll  m^{-a-1/2} (\log m)^{(d-1)(a+b)+1/2}.      
\end{equation}

\end{Lemma}
 \begin{proof} We prove the lemma for $m\asymp 2^nn^{d-1}$, $n\in \N$. Let $f\in \bW^{a,b}_A$.  
We approximate $f_l$ in $L_p$. By Theorem \ref{T2.5C} we obtain for $p\in [2,\infty)$
\begin{equation}\label{2.15}
\|f_l-G^p_{m_l}(f_l)\|_p \ll (\mb_l)^{-1/2}\|f_l\|_A \ll (\mb_l)^{-1/2}2^{-al}l^{(d-1)b}.
\end{equation}
We take $\mu\in (0,a)$ and specify
$$
m_l := [2^{n-\mu (l-n)}l^{d-1}],\quad l=n,n+1,\dots.
$$
In addition we include in the approximant 
$$
S_n(f) := \sum_{\|\bs\|_1\le n}\delta_\bs(f).
$$
Define
$$
A_m(f,p,\mu) := S_n(f)+\sum_{l>n} G_{m_l}^p(f_l).
$$
Then, we have built an $m$-term approximant of $f$ with 
$$
m\ll 2^nn^{d-1}  +\sum_{l\ge n} m_l \ll 2^nn^{d-1}.
$$
 The error   of this approximation in $L_p$ is bounded from above by
$$ 
\|f-A_m(f,p,\mu)\|_p \le \sum_{l\ge n} \|f_l-G^p_{m_l}(f_l)\|_p \ll \sum_{l\ge n} (\mb_l)^{-1/2}2^{-al}l^{(d-1)b}
$$
$$
\ll \sum_{l\ge n}2^{-1/2(n-\mu(l-n))}l^{-(d-1)/2}2^{-al}l^{(d-1)b} \ll 2^{-n(a+1/2)}n^{(d-1)(b-1/2)}.
$$
This completes the proof of lemma in the case $ 2\le p<\infty$.

 Let us discuss the case $p=\infty$. The proof repeats the proof in the above case $p<\infty$ with the following change.
Instead of using (\ref{2.14p}) for estimating an $m_l$-term approximation of $f_l$ in $L_p$ we use  (\ref{2.14'}) to estimate an  $m_l$-term approximation of $f_l$ in $L_\infty$. Then bound (\ref{2.15}) is replaced by 
\begin{equation}\label{2.16}
\|f_l-G^\infty_{m_l}(f_l)\|_\infty \ll (\mb_l)^{-1/2}(\ln 2^l)^{1/2}\|f_l\|_A \ll (\mb_l)^{-1/2}l^{1/2}2^{-al}l^{(d-1)b}.
\end{equation}
The extra factor $l^{1/2}$ in (\ref{2.16}) gives an extra factor $(\log m)^{1/2}$ in (\ref{2.14}). 
\end{proof}

We now complete the proof of Theorem \ref{2.8}. First, consider the case $1<q\le 2\le p<\infty$. It is well known (see, for instance, \cite{Tmon}, p.34, Theorem 2.1) that for $f\in \bW^r_q$ one has
\begin{equation}\label{2.9}
\|f_l\|_q \ll 2^{-lr}.
\end{equation}
Theorem \ref{A} implies
$$
\|f_l\|_A \ll 2^{-(r-1/q)l}l^{(d-1)(1-1/q)}.
$$
Therefore, it is sufficient to use Lemma \ref{L2.1} with $a=r-1/q$ and $b=1-1/q$ to obtain the upper bounds.

Second, the upper bounds in the case $2\le q\le p<\infty$ follow from the above case $1<q\le 2\le p<\infty$ with $q=2$. 
The lower bounds follow from Theorem \ref{T2.6} with $p=2$.  
The lower bounds in the case $2\le q\le p<\infty$ follow from known results for the case $1<p\le q<\infty$ in \cite{KTE1} (see Theorem \ref{baT2.3} below). 
\end{proof}
Let us discuss the case $p=\infty$. In the same way as Theorem \ref{T2.8} was derived from (\ref{2.13}) of Lemma \ref{L2.1}  the following upper bounds in case $p=\infty$ are derived from (\ref{2.14}) of Lemma \ref{L2.1}.
\begin{Theorem}\label{T2.9} We have
$$
 \sigma_m(\bW^r_q)_{\infty}
  \ll  \left\{\begin{array}{ll}   m^{-r+\eta}(\log m)^{(d-1)(r-2\eta)+1/2}, & 1<q\le 2,\quad r>1/q,\\ 
 m^{-r}(\log m)^{r(d-1)+1/2}, & 2\le q<\infty, \quad r>1/2.\end{array} \right.
$$
The upper bounds are provided by a constructive method $A_m(\cdot,\infty,\mu)$ based on greedy algorithms. 
\end{Theorem}

 Consider the case $\sigma_m(\bW^r_{1,\alpha})_p$, which is not covered by 
Theorems \ref{T2.6} and \ref{T2.8}. The function $F_r(\bx)$ belongs to the closure in $L_p$ of  $\bW^r_{1,\alpha}$, $r>1-1/p$, and, therefore, on the one hand
$$
\sigma_m(\bW^r_{1,\alpha})_p \ge \sigma_m(F_r(\bx))_p.
$$
On the other hand, it follows from the definition of $\bW^r_{1,\alpha}$ that for any $f\in \bW^r_{1,\alpha}$
$$
\sigma_m(f)_p \le \sigma_m(F_r(\bx))_p.
$$
Thus,
\be\label{2WF}
\sigma_m(\bW^r_{1,\alpha})_p = \sigma_m(F_r(\bx))_p.
\ee
We now prove some results on $\sigma_m(F_r(\bx))_p$.

\begin{Theorem}\label{T2.10} We have
$$
 \sigma_m(F_r(\bx))_p
  \asymp  \left\{\begin{array}{ll}   m^{-r+1-1/p}(\log m)^{(d-1)(r-1+2/p)}, & 1<p\le 2,\quad r>1-1/p,\\ 
 m^{-r+1/2}(\log m)^{r(d-1)}, & 2\le p<\infty, \quad r>1.\end{array} \right.
$$
The upper bounds are provided by a constructive method $A_m(\cdot,p,\mu)$ based on greedy algorithms. 
\end{Theorem}
\begin{proof} We begin with the case $1<p\le 2$. The following error bound for approximation by the hyperbolic cross polynomials is known (see, for instance, \cite{Tmon}, p.38)
\be\label{2.17}
E_{Q_n}(F_r)_p:=\inf_{t\in\Tr(Q_n)}\|F_r-t\|_p\ll 2^{-n(r-1+1/p)}n^{(d-1)/p}.
\ee
Taking into account that $|Q_n|\asymp 2^nn^{d-1}$ we obtain from (\ref{2.17}) the required upper bound in the case $1<p\le 2$. 
Thus, it remains to prove the matching lower bound in the case $1<p\le 2$. Denote
$$
\theta_n:=\{\bs\in \N^d:\|\bs\|_1=n\},\quad \Delta Q_n := \cup_{\bs\in\theta_n}\rho(\bs).
$$
Let $K_m:=\{\bk^j\}_{j=1}^m $ be given. Choose $n$ such that it is the minimal to satisfy
$$
|\Delta Q_n| \ge 4m.
$$
Clearly
$$
2^nn^{d-1} \asymp m.
$$
Denote
$$
\theta_n':=\{\bs\in\theta_n: |K_m\cap \rho(\bs)|\le |\rho(\bs)|/2\}.
$$
Note that for $\bs\in\N^d$ we have $|\rho(\bs)| = 2^n$. Then
$$
(|\theta_n|-|\theta_n'|)2^n/2 \le m\le |\theta_n|2^n/4,
$$
which implies
$$
|\theta_n'| \ge |\theta_n|/2.
$$
By Theorem \ref{T1.1}, (\ref{1.2}) with $q=2$, $1<p<2$, we obtain for any $t=\sum_{j=1}^m c_je^{i(\bk^j,\bx)}$
$$
\|F_r-t\|_p \gg \left(\sum_{\bs\in\theta_n'}\|\delta_\bs(F_r-t)\|_2^p 2^{n(p/2-1)}\right)^{1/p}
$$
$$
\gg \left(\sum_{\bs\in\theta_n'}2^{pn(-r+1/2)} 2^{n(p/2-1)}\right)^{1/p}\gg 2^{-n(r-1+1/p)}n^{(d-1)/p}.
$$
This gives the required lower bound for $1<p<2$. The above argument gives the lower bound in the case $p=2$ without use of Theorem \ref{T1.1} -- it is sufficient to use the Parseval identity. 

We now proceed to the case $2\le p<\infty$. Analysis here is similar to that in the proof of Theorem \ref{T2.8}. We get for
$$
F_r^l:=\sum_{\|\bs\|_1=l}\delta_\bs(F_r)
$$
$$
\|F_r^l\|_A \ll  2^{-lr}2^ll^{d-1}.
$$
The required upper bound follows from Lemma \ref{L2.1} with $a=r-1$ and $b=1$. 

The lower bound follows from the case $p=2$.
\end{proof}

In the same way as a modification of the proof of Theorem \ref{T2.8} gave Theorem \ref{T2.9}
the corresponding modification of the argument in the proof of Theorem \ref{T2.10} gives the following result.

\begin{Theorem}\label{T2.11} We have
$$
 \sigma_m(F_r(\bx))_\infty
  \ll    
 m^{-r+1/2}(\log m)^{r(d-1)+1/2},   \quad r>1. 
$$
The  bounds are provided by a constructive method $A_m(\cdot,\infty,\mu)$ based on greedy algorithms. 
\end{Theorem}

We now proceed to classes $\bH^r_q$ and $\bB^r_{q,\theta}$.
Define
$$
\|f\|_{\bH^r_q}:= \sup_\bs \|\delta_\bs(f)\|_q 2^{r\|\bs\|_1},
$$
and for $1\le \theta <\infty$ define
$$
\|f\|_{\bB^r_{q,\theta}}:= \left(\sum_{\bs}\left(\|\delta_\bs(f)\|_q 2^{r\|\bs\|_1}\right)^\theta\right)^{1/\theta}.
$$
We will write $\bB^r_{q,\infty}:=\bH^r_q$. 
With a little abuse of notation, denote the corresponding unit ball
$$
\bB^r_{q,\theta}:= \{f: \|f\|_{\bB^r_{q,\theta}}\le 1\}.
$$

It will be convenient for us to use the following slight modification of classes $\bB^r_{q,\theta}$. Define
$$
\|f\|_{\bH^r_{q,\theta}}:= \sup_n \left(\sum_{\bs:\|\bs\|_1=n}\left(\|\delta_\bs(f)\|_q 2^{r\|\bs\|_1}\right)^\theta\right)^{1/\theta}
$$
and 
$$
\bH^r_{q,\theta}:= \{f: \|f\|_{\bH^r_{q,\theta}}\le 1\}.
$$

The best $m$-term approximations of classes $\bB^r_{q,\theta}$ are studied in detail by A.S. Romanyuk \cite{Rom1}. 
The following theorem was proved in \cite{T29} (see also \cite{Tmon}, Ch.4). The proofs from \cite{T29} and \cite{Tmon} are constructive.

\begin{Theorem}\label{T2.7} Let $1<q\le p\le 2$, $r>\beta$. Then
$$
\sigma_m(\bH^r_{q})_p \asymp m^{-r+\beta}(\log m)^{(d-1)(r-\beta+1/p))}.
$$
\end{Theorem}

  The following analog of Theorem \ref{T2.8} for classes $\bH^r_q$ was proved in \cite{Rom1}. The proof in \cite{Rom1} in the case $p>2$ is not constructive.
\begin{Theorem}\label{T2.12} One has
$$
 \sigma_m(\bH^r_{q})_{p}
  \asymp  \left\{\begin{array}{ll} m^{-r+\beta}(\log m)^{(d-1)(r-\beta+1/p)}, & 1<q\le p\le 2,\quad r>\beta,\\
 m^{-r+\eta}(\log m)^{(d-1)(r-1/q+1)}, & 1<q\le 2\le p<\infty,\quad r>1/q,\\ 
 m^{-r}(\log m)^{(d-1)(r+1/2)}, & 2\le q\le p<\infty, \quad r>1/2.\end{array} \right.
$$
\end{Theorem}
\begin{Proposition}\label{HCp} The upper bounds in Theorem \ref{T2.12} are provided by a constructive method $A_m(\cdot,p,\mu)$ based on greedy algorithms. 
\end{Proposition}
\begin{proof}The case $p\le 2$, which corresponds to the first line, follows from Theorem \ref{T2.7}. We now consider $p\ge 2$.  We get from the definition of classes $\bH^r_q$ for $1<q<\infty$  :
$$
 f\in \bH^r_q\quad \iff\quad \|\delta_\bs(f)\|_q \le 2^{-r\|\bs\|_1}.
$$
Next,
$$
\|\delta_\bs(f)\|_A \ll 2^{\|\bs\|_1/q}\|\delta_\bs(f)\|_q.
$$
Therefore, for $f\in\bH^r_q$ we obtain
$$
\|f_l\|_A \ll 2^{-(r-1/q)l}l^{d-1}.
$$
Applying Lemma \ref{L2.1} with $a=r-1/q$, $b=1$ in the case $1<q\le2\le p<\infty$, we obtain required upper bounds. 
The upper bounds in the case $2\le q\le p<\infty$ follow from the above case with $q=2$. 
The lower bounds in the case $1<q\le2\le p<\infty$ follow from Theorem \ref{T2.7}. 

The lower bounds in the case $2\le q\le p<\infty$ follow from known results in \cite{KTE1} (see Theorem \ref{baT2.5} below). 

\end{proof}

In the case $p=\infty$ we have.

\begin{Theorem}\label{T2.13} We have
$$
 \sigma_m(\bH^r_{q})_{\infty}
  \ll  \left\{\begin{array}{ll}   m^{-r+\eta}(\log m)^{(d-1)(r-1/q+1)+1/2}, & 1<q\le 2,\quad r>1/q,\\ 
 m^{-r}(\log m)^{(r+1/2)(d-1)+1/2}, & 2\le q<\infty, \quad r>1/2.\end{array} \right.
$$
The upper bounds  are provided by a constructive method $A_m(\cdot,\infty,\mu)$ based on greedy algorithms. 
\end{Theorem}

We now proceed to classes $\bB^r_{q,\theta}$. There is the following extension  of Theorem \ref{T2.12} (see \cite{Rom1}).
\begin{Theorem}\label{T2.12B} One has
$$
 \sigma_m(\bB^r_{q,\theta})_{p} \asymp  \sigma_m(\bH^r_{q,\theta})_{p} 
 $$
 $$
  \asymp  \left\{\begin{array}{ll} m^{-r+\beta}(\log m)^{(d-1)(r-\beta+1/p-1/\theta)}, & 1<q\le p\le 2,\quad r>\beta,\\
 m^{-r+\eta}(\log m)^{(d-1)(r-1/q+1-1/\theta)}, & 1<q\le 2\le p<\infty,\quad r>1/q,\\ 
 m^{-r}(\log m)^{(d-1)(r+1/2-1/\theta)}, & 2\le q\le p<\infty, \quad r>1/2.\end{array} \right.
$$
\end{Theorem}
\begin{Proposition}\label{BCp} The upper bounds in Theorem \ref{T2.12B} are provided by a constructive method based on greedy algorithms. 
\end{Proposition}
\begin{proof} In the same way as Theorems \ref{T2.6} and \ref{T2.7} were proved in \cite{Tmon} one can prove the following result for classes $\bB^r_{q,\theta}$ and $\bH^r_{q,\theta}$ in the case $1\le q\le p \le 2$, $p>1$:
$$
 \sigma_m(\bB^r_{q,\theta})_{p} \asymp  \sigma_m(\bH^r_{q,\theta})_{p} \asymp \left(\frac{(\log m)^{d-1}}{m}\right)^{r-\beta}(\log m)^{(d-1)(1/p-1/\theta)}.
 $$
This proves the statement of Proposition \ref{BCp} for the first relation in Theorem \ref{T2.12B}.

 We now consider $p\ge 2$.  We get from the definition of classes $\bH^r_{q,\theta}$ for $1<q<\infty$  :
$$
 f\in \bH^r_{q,\theta}\quad \iff\quad \left(\sum_{\|\bs\|_1=l}\|\delta_\bs(f)\|_q^\theta\right)^{1/\theta} \le 2^{-rl}.
$$
Next,
$$
\|f_l\|_A\le \sum_{\|\bs\|_1=l}\|\delta_\bs(f)\|_A \ll 2^{l/q}\sum_{\|\bs\|_1=l}\|\delta_\bs(f)\|_q
$$
$$
\ll 2^{l/q}l^{(d-1)(1-1/\theta)}\sum_{\|\bs\|_1=l}\left(\|\delta_\bs(f)\|_q^\theta\right)^{1/\theta}\ll 
2^{-l(r-1/q)}l^{(d-1)(1-1/\theta)}.
$$
Therefore, for $f\in\bH^r_{q,\theta}$ we obtain
$$
\|f_l\|_A \ll 2^{-(r-1/q)l}l^{(d-1)(1-1/\theta)}.
$$
Applying Lemma \ref{L2.1} with $a=r-1/q$ and $b=1-1/\theta$ in the case $1<q\le2\le p<\infty$, we obtain required upper bounds. 
The upper bounds in the case $2\le q\le p<\infty$ follow from the above case with $q=2$. 
The lower bounds in the case $1<q\le2\le p<\infty$ follow from the case $1<q\le 2$, $p=2$.

It was proved in \cite{KTE1} that 
\be\label{*} 
\sigma_m(\bH^r_\infty\cap\Tr(\Delta Q_n))_p \gg m^{-r} (\log m)^{(d-1)(r+1/2)}
\ee
with some $n$ such that $m\asymp 2^nn^{d-1}$. It is easy to see that for any $f \in  \bH^r_\infty\cap\Tr(\Delta Q_n))_p$ we have
\be\label{**}
\|f\|_{\bB^r_{q,\theta}} \ll n^{(d-1)/\theta}.
\ee
Relations (\ref{*}) and (\ref{**}) imply the lower bound in the case $2\le q\le p<\infty$.

\end{proof}

In the case $p=\infty$ we have.

\begin{Theorem}\label{T2.13} We have
$$
 \sigma_m(\bH^r_{q,\theta})_{\infty}
  \ll  \left\{\begin{array}{ll}   m^{-r+\eta}(\log m)^{(d-1)(r-1/q+1-1/\theta)+1/2}, & 1<q\le 2,\quad r>1/q,\\ 
 m^{-r}(\log m)^{(r+1/2-1/\theta)(d-1)+1/2}, & 2\le q<\infty, \quad r>1/2.\end{array} \right.
$$
The upper bounds in Theorem \ref{T2.13} are provided by a constructive method $A_m(\cdot,\infty,\mu)$ based on greedy algorithms. 
\end{Theorem}

We formulate some known results in the case $1<p\le q\le\infty$.
\begin{Theorem}\label{baT2.3} Let $1<p\le q<\infty$, $r>0$. Then
$$
\sigma_m(\bW^r_q)_p \asymp m^{-r}(\log m)^{(d-1)r}.
$$
\end{Theorem}
The upper bound in Theorem \ref{baT2.3} follows from error bounds for approximation by the hyperbolic cross polynomials (see \cite{Tmon}, Ch.2, \S2)
$$
E_{Q_n}(\bW^r_q,L_q) \ll 2^{-rn},\quad 1<q<\infty.
$$
The lower bound in Theorem \ref{baT2.3} was proved in \cite{KTE1}.

The following result for $\bH^r_q$ classes is known.
\begin{Theorem}\label{baT2.5} Let $p\le q$, $2\le q\le \infty$, $1<p<\infty$, $r>0$. Then
$$
\sigma_m(\bH^r_{q})_p \asymp m^{-r}(\log m)^{(d-1)(r+1/2)}.
$$
\end{Theorem}
The lower bound for all $p>1$
$$
\sigma_m(\bH^r_\infty)_p \gg m^{-r}(\log m)^{(d-1)(r+1/2)}
$$
was obtained in \cite{KTE1}. The matching upper bounds follow from approximation by the hyperbolic cross polynomials (see \cite{Tmon}, Ch.2, Theorem 2.2)
$$
E_{Q_n}(\bH^r_q)_q:=\sup_{f\in\bH^r_q}E_{Q_n}(f)_q \asymp n^{(d-1)/2}2^{-rn},\quad 2\le q<\infty.
$$
The following result for $\bB$ classes was proved in \cite{Rom1}.

\begin{Theorem}\label{baT2.5B} Let $1<p\le q<\infty$, $2\le q< \infty$, $1<p<\infty$, $r>0$. Then
$$
\sigma_m(\bB^r_{q,\theta})_p \asymp m^{-r}(\log m)^{(d-1)(r+1/2-1/\theta)_+}.
$$
\end{Theorem}

\section{Application of the entropy numbers}

Let $X$ be a Banach space and let $B_X$ denote the unit ball of $X$ with the center at $0$. Denote by $B_X(y,r)$ a ball with center $y$ and radius $r$: $\{x\in X:\|x-y\|\le r\}$. For a compact set $A$ and a positive number $\e$ we define the covering number $N_\e(A)$
 as follows
$$
N_\e(A) := N_\e(A,X)  
:=\min \{n : \exists y^1,\dots,y^n :A\subseteq \cup_{j=1}^n B_X(y^j,\e)\}.
$$
  
For a compact $A$ we define an $\e$-distinguishable set $\{x^1,\dots,x^m\} \subseteq A$ as a set with the property
\begin{equation}\label{31.4}
\|x^i-x^j\| >\e, \quad \text{for all}\quad i,j: i\neq j. 
\end{equation}
Denote by $M_\e(A):=M_\e(A,X)$ the maximal cardinality of $\e$-distinguishable sets of a compact $A$. The following simple theorem is well known.
\begin{Theorem}\label{Theorem 33.1} For any compact set $A$ we have
\begin{equation}\label{31.5}
M_{2\e}(A)\le N_\e(A)\le M_\e(A).  
\end{equation}
\end{Theorem}

Consider the entropy numbers $\e_k(A,X)$:
$$
\e_k(A,X) :=  \inf \{\e : \exists y^1,\dots ,y^{2^k} \in X : A \subseteq \cup_{j=1}
^{2^k} B_X(y^j,\e)\}.
$$

The following theorem is from \cite{Tappr}.
\begin{Theorem}\label{T3.1} Let a compact $F\subset X$ be such that there exists a normalized system $\D$, $|\D|=N$, and  a number $r>0$ such that 
$$
  \sigma_k(F,\D)_X \le k^{-r},\quad m\le N.
$$
Then for $k\le N$
\begin{equation}\label{3.0}
\e_k(F,X) \le C(r) \left(\frac{\log(2N/k)}{k}\right)^r.
\end{equation}
\end{Theorem}

We use the above theorem to prove the following lower bound for 
best $m$-term approximations.
\begin{Theorem}\label{T3.4} In the case $d=2$ the following lower bound holds for any $q<\infty$, $r>1/q$
$$
\sigma_m(\bW^r_q)_\infty \gg m^{-r} (\log m)^{1/2}.
$$
\end{Theorem}
\begin{proof}   
 We will use a special inequality from \cite{TE3}, which is called the Small Ball Inequality. 
 For an even number $n$ define
$$
Y_n:=\{\bs=(2n_1,2n_2),\quad n_1+n_2=n/2\}.
$$
Then for any coefficients $\{c_\bk\}$
\begin{equation}\label{2.6.3}
\|\sum_{\bs\in Y_n} \sum_{\bk\in \rho(\bs)}c_\bk e^{i(\bk,\bx)}\|_\infty \ge C\sum_{\bs\in Y_n}\|\sum_{\bk\in\rho(\bs)} c_\bk e^{i(\bk,\bx)}\|_1,
\end{equation}
where $C$ is a positive number. Inequality (\ref{2.6.3}) plays a key role in the proof of lower bounds for the entropy numbers.

Take any even $n\in\N$, which will be chosen, depending on $m$, later. Consider the following compact
$$
F(Y_n)_\infty :=\{t=\sum_{\bs\in Y_n}t_\bs : t_\bs \in \Tr(\rho(\bs)),\quad \|t_\bs\|_\infty \le 1\}.
$$
The known results on volumes of sets of Fourier coefficients of trigonometric polynomials imply the following lemma (see \cite{TE2} and \cite{TE3}).
\begin{Lemma}\label{L3.2} There exist $2^{n2^{n-1}}$ functions $f_j\in F(Y_n)_\infty$, $j=1,\dots,2^{n2^{n-1}}$ such that for $i\neq j$
$$
\|f_i-f_j\|_2 \gg n^{1/2}.
$$
\end{Lemma} 
We now show that for $f_j$ from Lemma \ref{L3.2} we have
$$
\|f_i-f_j\|_\infty \gg n .
$$ 
Indeed, for any $f\in F(Y_n)_\infty$ we have
$$
\|f\|_2^2 = \sum_{\bs\in Y_n} \|t_\bs\|_2^2 \le  \sum_{\bs\in Y_n} \|t_\bs\|_1\|t_\bs\|_\infty \le \sum_{\bs\in Y_n} \|t_\bs\|_1.
$$
It remains to apply the Small Ball Inequality (\ref{2.6.3}).
Therefore, for $k=n2^{n-1}$ we get by Theorem \ref{Theorem 33.1}
 $$
 \ep_k(F(Y_n)_\infty)_\infty \gg \log k.
 $$
 We now use Theorem \ref{T3.1}. We specify $F:=F(Y_n)_\infty$, $\D:=\{e^{i(\bk,\bx)}: \|\bk\|_\infty \le 2^{n+1}\}$, $X:=L_\infty$. It is clear that for $l\ge \dim T(Y_n) \asymp n2^n \asymp k$
 we have 
 $$
 \sigma_l(F,\D) =0.
 $$
 Also, for any $f\in F$ we have
 $$
 \|f\|_\infty \le n/2 \ll \log k.
 $$
 Denote 
 $$
 B:= \max_l l^r \sigma_l(F,\D)_\infty.
 $$
 By Theorem \ref{T3.1} we obtain
 $$
 \log k \ll B n^r k^{-r} \quad \text{and}\quad B\gg n^{-r}k^r\log k.
 $$
 This implies that there is $l\asymp k$ such that 
 $$
 \sigma_l(F,\D)_\infty \gg n^{-r}\log k \asymp (\log k)^{1-r}.
 $$
 Next, it is clear that for any $m$
 $$
 \sigma_m(F,\D)_\infty \ll \sigma_m(F,\Tr)_\infty.
 $$
 Further, by Littlewood-Paley theorem there is $c_1(q)>0$ such that
 $$
 c_1(q)n^{-1/2}2^{-rn} F \subset \bW^r_q,\quad q<\infty.
 $$
 This completes the proof. 
 
 \end{proof}
 
\section{Numerical Integration} 

\subsection{Notations. The problem setting}
Numerical integration seeks good ways of approximating an integral
$$
\int_\Omega f(\bx)d\mu
$$
by an expression of the form
\be\label{6.1}
\Lambda(f,X_m) :=\sum_{j=1}^m\la_jf(\xi^j), 
\ee
 where $X_m:=\{\xi^1,\dots,\xi^m\},\quad \xi^j \in \Omega,\quad j=1,\dots,m,  
$. For  a function class $\bW$ denote
$$
\Lambda(\bW,X_m):=\sup_{f\in\bW} |\int_\Omega f(\bx)d\mu - \Lambda(f,X_m)|.
$$
We are interested in dependence on $m$ of the best $m$-knot error of numerical integration
$$
\delta_m (\bW) := \inf_{ \lambda_1,\dots,\lambda_m; \xi^{1},\dots,
\xi^m}\Lambda(\bW,X_m)
$$
for some function classes $\bW$.  

\subsection{Known lower bounds}
The reader can find results and historical comments on numerical integration of classes of functions with mixed smoothness in the book \cite{TBook}, Ch.4 and the survey paper \cite{T11}.
The following theorem was proved in \cite{Tem22}.
\begin{Theorem}\label{T6.2.1} The following lower estimate is valid for any
cubature formula $(\Lambda,X_m)$ with $m$ knots $(r > 1/p)$
$$
\Lambda(\bW_{p}^r,X_m) \ge C(r,d,p)m^{-r}
(\log m)^{\frac{d-1}{2}},\qquad 1 \le p < \infty .
$$
\end{Theorem}
The proof of this theorem is based on Theorem \ref{T1.2} from Introduction.   
Theorem \ref{T1.2} is used to prove the following assertion.

\begin{Lemma}\label{L6.2.2} Let the coordinates of the vector $ \bs$ be natural
numbers and $\| \bs\|_1 = n$. Then for any $N\le 2^{n-1}$ and an
arbitrary cubature formula $(\Lambda,X_N)$ with $N$ knots there is a
$t_{ \bs}\in \Tr(2^{ \bs} ,d)$ such that
$\|t_{ \bs}\|_{\infty} \le 1$ and
\be\label{6.2.1}
\hat t_{ \bs} (0) - \Lambda(t_{ \bs},X_N) \ge C(d) > 0.
\ee
\end{Lemma}
For a given $m$  
choose $n$ such that
$$
 m \le 2^{n-1}  < 2m.
$$
Consider the polynomial
$$
 t( \bx) =\sum_{\| \bs\|_1=n}t_{ \bs} ( \bx) ,
$$
where $t_{ \bs}$ are polynomials from Lemma \ref{L6.2.2} with $N = m$.
Then
\be\label{6.2.2}
\hat t( 0) - \Lambda(t,X_m)\ge C(d)n^{d-1}.
\ee

The proof of Theorem \ref{T6.2.1} was completed by establishing that
\be\label{6.2.3}
\|t\|_{\bW^r_p} \ll \|t\|_{\bB^r_{p,2}} \ll 2^{rn} n^{(d-1)/2}.
\ee
Theorem \ref{T6.2.1} gives the same lower bound for different parameters $1\le p<\infty$. It is clear that the bigger the $p$ the stronger the statement. 

\subsection{New lower bounds}

We obtain lower bounds for numerical integration with respect to a special class of knots. 
Let $\bs = (s_1,\dots,s_d)$, $s_j\in \N_0$, $j=1,\dots,d$. We associate with $\bs$ a web
$W(\bs)$ as follows: denote 
$$
w(\bs,\bx) := \prod_{j=1}^d \sin (2^{s_j}x_j)
$$
and define
$$
W(\bs) := \{\bx: w(\bs,\bx)=0\}.
$$
\begin{Definition}\label{D4.1} We say that a set of knots $X_m:=\{\xi^i\}_{i=1}^m$ is an $(n,l)$-net if $|X_m\setminus W(\bs)| \le 2^l$ for all $\bs$ such that $\|\bs\|_1=n$.
\end{Definition}

\begin{Theorem}\label{T4.2} For any cubature formula $(\Lambda,X_m)$ with respect to a $(n,n-1)$-net $X_m$ we have
$$
\Lambda(\bW^r_p,X_m) \gg 2^{-rn}n^{(d-1)/2},\quad 1\le p<\infty.
$$
\end{Theorem}
\begin{proof} This proof is similar to that of Theorem \ref{T6.2.1} from \cite{Tem22}. Take $\bs$ such that $\|\bs\|_1=n$ and consider $\Tr(\bN,d)$ with $N_j:=2^{s_j-1}$, $j=1,\dots,d$. Then
$$
\dim\Tr(\bN,d) \ge 2^{\|\bs\|_1} = 2^n.
$$
Let $I(\bs)$ be a set of indexes such that
$$
X_m\setminus W(\bs) = \{\xi^i\}_{i\in I(\bs)}.
$$
Then by our assumption $|I(\bs)| \le 2^{n-1}$. Consider
$$
\Psi(\bs):=\{t\in \Tr(\bN,d): t(\xi_i)=0, i\in I(\bs)\}.
$$
Then $\dim\Psi(\bs) \ge 2^{n-1}$. By Theorem \ref{T1.2} we find $t_\bs^1 \in \Psi(\bs)$ such that $\|t_\bs^1\|_\infty =1$ and $\|t_\bs^1\|_2 \ge c(d)>0$. Consider
$$
t(\bx):=\sum_{\|\bs\|_1=n} t_\bs(\bx),\qquad t_\bs(\bx):= |t_\bs^1(\bx)|^2w(\bs,\bx)^2.
$$
We have
$$
t_\bs^1(\bx)w(s_1,x_1) 
$$
$$
= (2i)^{-1}\left(\sum_{\bk:|k_j|\le 2^{s_j-1}} {\hat t}_\bs^1(\bk)e^{i(\bk,\bx)+i2^{s_1}x_1} + \sum_{\bk:|k_j|\le 2^{s_j-1}} {\hat t}_\bs^1(\bk)e^{i(\bk,\bx)-i2^{s_1}x_1}\right).
$$
and
$$
\|t_\bs^1(\bx)w(s_1,x_1)\|_2^2 =  2^{-1}\|t_\bs^1\|_2^2.
$$
Therefore,
$$
\|t_\bs^1(\bx)w(\bs,\bx)\|_2^2 =  2^{-d}\|t_\bs^1\|_2^2 \ge c_1(d) >0.
$$
Then relation (\ref{6.2.2}) is obviously satisfied for our $t$. Relation (\ref{6.2.3}) is proved in the same way as it was proved in \cite{Tem22}.
\end{proof}

   The example that was constructed in the proof of Theorem \ref{T6.2.1} (see above) provides the lower bound for the Besov-type classes. 
   Other proof of Theorem \ref{T6.2.3} is given in \cite{DU}. 
     
\begin{Theorem}\label{T6.2.3} The following lower estimate is valid for any
cubature formula $(\Lambda,X_m)$ with $m$ knots $(r > 1/p)$
$$
\Lambda(\bB_{p,\theta}^r,X_m) \ge C(r,d,p)m^{-r}
(\log m)^{(d-1)(1-1/\theta)},\quad 1 \le p \le \infty,\quad 1\le \theta\le\infty .
$$
\end{Theorem}
Indeed, the proof of (\ref{6.2.3}) from \cite{Tem22} implies
\be\label{6.2.4}
  \|t\|_{\bB^r_{p,\theta}} \ll 2^{rn} n^{(d-1)/\theta}.
\ee

In the same way the proof of Theorem \ref{T4.2} gives the following result.

\begin{Theorem}\label{T4.4} For any cubature formula $(\Lambda,X_m)$ with respect to a $(n,n-1)$-net $X_m$ we have
$$
\Lambda(\bB_{p,\theta}^r,X_m) \gg 2^{-rn}n^{(d-1)(1-1/\theta)},\quad 1\le p\le\infty.
$$
\end{Theorem} 

We note that Theorems \ref{T4.2} and \ref{T4.4} provide lower bounds for numerical integration with respect to sparse grids and their modifications. For $n\in\N$ we define the sparse grid $SG(n)$ as follows
$$
SG(n) := \{\xi(\bn,\bk) = (\pi k_12^{-n_1},\dots,\pi k_d 2^{-n_d}),
$$
$$
 0\le k_j<2^{n_j}, j=1,\dots,d,\quad \|\bn\|_1=n\}.
$$
Then it is easy to check that $SG(n)\subset W(\bs)$ with any $\bs$ such that $\|\bs\|_1=n$. Indeed, let $\xi(\bn,\bk)\in SG(n)$. Take any $\bs$ with $\|\bs\|_1=n$. Then $\|\bs\|_1=\|\bn\|_1$ and there exists $j$ such that $s_j\ge n_j$. For this $j$ we have
$$
\sin 2^{s_j}\xi(\bn,\bk)_j = \sin 2^{s_j}\pi k_j 2^{-n_j} =0\quad \text{and} \quad w(\bs,\xi(\bn,\bk)=0.
$$
This means that $SG(n)$ is an $(n,l)$-net for any $l$.
We note that $|SG(n)|\asymp 2^n n^{d-1}$. It is known (see \cite{T23}) that there exists a cubature formula $(\Lambda,SG(n))$ such that
\be\label{4.6}
\Lambda(\bH^r_p,SG(n)) \ll 2^{-rn}n^{d-1},\quad 1\le p\le \infty,\quad r>1/p.
\ee
Theorem \ref{T4.4} with $\theta=\infty$ shows that the bound (\ref{4.6}) is sharp. Moreover, 
Theorem \ref{T4.4} shows that even an addition of extra $2^{n-1}$ arbitrary knots to $SG(n)$ will not improve the bound in (\ref{4.6}). In the case $X_m=SG(n)$ other proof of Theorem \ref{T4.4} is given in \cite{DU}. 

\section{Approximate recovery}

Consider the following recovering operator. For fixed $m$, $X_m:=\{\xi^j\}_{j=1}^m$, and 
$\psi_1(\bx),\dots,\psi_m(\bx)$ define the linear operator
$$
\Psi(f,X_m):=\sum_{j=1}^m f(\xi^j)\psi_j(\bx).
$$
For a function class $\bW$ define
$$
\Psi(\bW,X_m)_p:=\sup_{f\in \bW}\|f-\Psi(f,X_m)\|_p.
$$
The main result of this section is the following theorem.
\begin{Theorem}\label{T5.1} For any recovering operator $\Psi(\cdot,X_m)$ with respect to a\newline $(n,n-1)$-net $X_m$ we have for $1\le q< p<\infty$
$$
\Psi(\bH^r_q,X_m)_p \gg 2^{-n(r-\beta)}n^{(d-1)/p},\quad \beta:=1/q-1/p.
$$
\end{Theorem}
Before proceeding to the proof of this theorem we make some historical comments. The problem of optimal recovery on classes of functions with mixed smoothness is wide open. Denote for a class $\bW$
$$
\varrho_m(\bW)_p := \inf_{X_m;\psi_1,\dots,\psi_m} \Psi(\bW,X_m)_p.
$$
The right order of this characteristic is known only in a few cases. It was established in \cite{VT51} that
\be\label{5.1}
\varrho_m(\bW^r_2)_\infty \asymp m^{-r+1/2} (\log m)^{r(d-1)}.
\ee
The upper bound in (\ref{5.1}) was obtained by recovering with the Smolyak-type operator
$T_n$ with appropriate $n$. The operator $T_n$ uses the sparse grid $SG(n+d)$ and $\psi_j \in \Tr(Q_{n+d})$. The operator $T_n$ and its variants were studied in many papers. 
It is proved in \cite{T23} that for any $f\in \bH^r_p$, $1\le p \le \infty$, $r>1/p$
\be\label{5.2}
\|f-T_n(f)\|_p \ll 2^{-rn} n^{d-1}.
\ee
The following bound is obtained in \cite{TBook}, Ch.4, Section 5, Remark 2: for any $f\in \bH^r_q$, $1\le q< p < \infty$, $r>1/q$
\be\label{5.3}
\|f-T_n(f)\|_p \ll 2^{-n(r-\beta)} n^{(d-1)/p}.
\ee
The upper bound (\ref{5.2}) and the lower bound for the Kolmogorov width from \cite{VT59}: for $d=2$
$$
d_m(\bH^r_\infty,L_\infty)\asymp m^{-r} (\log m)^{r+1}
$$
imply for $d=2$
\be\label{5.4}
\varrho_m(\bH^r_\infty)_\infty \asymp m^{-r} (\log m)^{r+1}.
\ee
The upper bound (\ref{5.3}) and known bounds for the Kolmogorov width (see, for instance, \cite{Tmon}, Ch.3): for $1\le q\le 2$, $r>1/q$
$$
d_m(\bH^r_q,L_2)\asymp m^{-r+\eta} (\log m)^{(d-1)(r+1-1/q)},\quad \eta:=1/q-1/2.
$$
imply for $1\le q\le 2$, $r>1/q$
\be\label{5.4}
\varrho_m(\bH^r_q)_2 \asymp m^{-r+\eta} (\log m)^{(d-1)(r+1-1/q)}.
\ee
As we already pointed out above $T_n(f)\in \Tr(Q_{n+d})$. Thus $T_n$ provides approximation from the hyperbolic cross $\Tr(Q_{n+d})$. The upper bound (\ref{5.3}) and known bounds for the best hyperbolic cross approximation (see, for instance, \cite{Tmon}, Ch.2, Theorem 2.2) show that for $1\le q<p <\infty$, $r>1/q$, the operator $T_n$ provides 
optimal in the sense of order rate of approximation for $\bH^r_q$ in the $L_p$. 

{\bf Proof of Theorem \ref{T5.1}.} We use the polynomials $t^1_\bs$ constructed in the proof of Theorem \ref{T4.2}. We also need some more constructions. Let 
 $$
{\mathcal K}_{N} (x) :=
\sum_{|k|\le N} \bigl(1 - |k|/N\bigr) e^{ikx} =\bigl(\sin (Nx/2)\bigr)^2\bigm /\bigl(N (\sin (x/2)\bigr)^2\bigr)
$$
be a univariate Fej\'er kernel.
The Fej\'er kernel ${\mathcal K}_{N}$ is an even nonnegative trigonometric
polynomial in $\Tr(N-1)$.  
From the obvious relations
$$
\| {\mathcal K}_{N} \|_1 = 1, \qquad \| {\mathcal K}_{N} \|_{\infty} = N
$$
and the inequality  
$$
\| f \|_q \le \| f \|_1^{1/q} \| f \|_{\infty}^{1-1/q}
$$
we get  
\begin{equation}\label{2.1.8}
C N^{1-1/q}\le \| {\mathcal K}_{N} \|_q \le N^{1-1/q}, \qquad
 1\le q\le \infty.
\end{equation}
In the multivariate case define
$$
\mathcal K_{\mathbf N} (\mathbf x) :=\prod_{j=1}^d\mathcal K_{N_j}  (x_j)  ,\qquad
\mathbf N = (N_1 ,\dots,N_d).
$$
Then the $\mathcal K_{\mathbf N}$ are nonnegative trigonometric polynomials from  $\Tr(\mathbf N-\mathbf 1,d)$  which  have
the following properties:
\begin{equation}\label{2.2.6}
\|\mathcal K_{\mathbf N}\|_1  = 1,
\end{equation}
\begin{equation}\label{2.2.7}
\|\mathcal K_{\mathbf N}\|_q\asymp \vartheta(\mathbf N)^{1-1/q},\qquad
1\le q\le\infty.
\end{equation}
 For $n$ of the form $n=4l$, $l\in \N$, define
 $$
 Y(n,d):=\{\bs: \bs=(4l_1,\dots,4l_d),\quad l_1+\dots+l_d=n/4,\quad l_j\in\N,\quad j=1,\dots,d\}.
 $$
 Define for $\bs\in Y(n,d)$
 $$
 t_\bs(\bx) := t^1_\bs(\bx)\K_{2^{\bs-2}}(\bx-\bx^*),
 $$
 where $\bx^*$ is a point of maximum of $|t^1_\bs(\bx)|$. Finally, define
 $$
 t(\bx) := \sum_{\bs\in Y(n,d)} t_\bs(\bx)w(\bs,\bx).
 $$
 Then we have
 $$
 |t_\bs(\bx^*| \gg 2^n
 $$
 and, therefore, by Nikol'skii's inequality
 $$
 \|t_\bs\|_2 \gg 2^{n/2}.
 $$
 It follows from our definition of $Y(n,d)$ that polynomials $t_\bs(\bx)w(\bs,\bx)$, $\bs\in Y(n,d)$, form an orthogonal system. This implies
 \be\label{5.10}
 \|t\|_2^2 \gg 2^n n^{d-1}.
 \ee
 Take any $p\in (1,\infty)$ and by Theorem \ref{T1.1} estimate
 \be\label{5.11}
 \|t\|_{p'}\ll \left(\sum_{\bs\in Y(n,d)} \|t_\bs\|_1^{p'}2^{\|\bs\|_1(p'-1)}\right)^{1/p'} \ll 2^{n/p}n^{(d-1)/p'}.
 \ee
 Relations (\ref{5.10}) and (\ref{5.11}) imply
 \be\label{5.12}
 \|t\|_{p}\gg   2^{n(1-1/p)}n^{(d-1)/p}.
 \ee
It is clear that
\be\label{5.13}
\|t\|_{\bH^r_q}\ll 2^{n(r+1-1/q)}.
\ee
Bounds (\ref{5.12}) and (\ref{5.13}) imply the required in Theorem \ref{T5.1} bound. 
Theorem \ref{T5.1} is proved.

The inequaality 
$$
\|t\|_{\bB^r_{q,\theta}}\ll 2^{n(r+1-1/q)}n^{(d-1)/\theta}
$$
and (\ref{5.12}) imply the following result.

\begin{Theorem}\label{T5.2} For any recovering operator $\Psi(\cdot,X_m)$ with respect to a \newline $(n,n-1)$-net $X_m$ we have for $1\le q< p<\infty$, $r>\beta$,
$$
\Psi(\bB^r_{q,\theta},X_m)_p \gg 2^{-n(r-\beta)}n^{(d-1)(1/p-1/\theta)},\quad \beta:=1/q-1/p.
$$
\end{Theorem}

\section{Discussion}

As we stressed in the title and in the Introduction we are interested in constructive methods of $m$-term approximation with respect to the trigonometric system. Theorem \ref{TC}, basically, solves this problem for approximation in $L_p$, $2\le p<\infty$. We do not have a similar result for approximation in $L_p$, $1<p<2$. The corresponding Lebesgue-type inequality from \cite{T144} gives for $1<p<2$
$$
\|f_{C(t,p,d)m^{p'-1}\log (m+1)}\|_p \le C\sigma_m(f)_p,
$$
which is much weaker than Theorem \ref{TC}. It would be interesting to obtain good Lebesgue-type inequalities in the case $1<p<2$ for either the WCGA or for some other constructive methods. 

Main results of this paper are on the $m$-term approximation in the case $2\le p\le\infty$. For $p\in[2,\infty)$ the situation is very good: we have a universal algorithm (WCGA), which provides almost optimal (up to extra $(\log m)^{C(r,d)}$ factor) $m$-term approximation for all classes $\bW^r_q$, $\bH^r_q$, and $\bB^r_{q,\theta}$. Also, there are constructive methods, based on greedy algorithms, which provide the optimal rate for the above classes. 
However, the upper bounds in, say, Theorem \ref{T2.8C} hold for smoothness $r$ larger than the one required for embedding of $\bW^r_q$ into $L_p$. It would be interesting to find constructive methods, which provide right orders of decay of $\sigma_m(\bW^r_q)_p$,  $\sigma_m(\bH^r_q)_p$, and $\sigma_m(\bB^r_{q,\theta})_p$ for small smoothness. 

The case $p=\infty$ (approximation in the uniform norm) is a very interesting and difficult case. The space $C(\T^d)$ is not a smooth Banach space. Therefore, the existing greedy approximation theory does not apply directly in the case of approximation in $L_\infty$. In particular, there is no analog of Theorem \ref{TC} in the case $p=\infty$. However, for the function classes with mixed smoothness there is a way around this problem. As it is demonstrated in the proof of Theorem \ref{T2.4} we can use greedy algorithms in $L_p$ with large $p$ to obtain bounds on $m$-term approximation in $L_\infty$. The price we pay for this trick is an extra $(\log m)^{1/2}$ factor in the error bound. This extra factor results from the factor $p^{1/2}$ in the error bounds of approximation by greedy algorithms in $L_p$, $2\le p<\infty$ (see Remark \ref{R2.1} and Theorem \ref{T2.2}). An extra $(\log m)^{1/2}$ appears, as a result of different techniques, in other upper bounds of asymptotic characteristics of classes of functions with mixed smoothness, when we go from $p<\infty$ to $p=\infty$ (see, for instance, \cite{Tbook}, Ch.3, Section 3.6). Unfortunately, we do not have matching lower bounds for our upper bounds for $m$-term approximation in $L_\infty$. A very special case in Theorem \ref{T3.4} could be interpreted as a hint that we cannot get rid of that extra $(\log m)^{1/2}$ for approximation in $L_\infty$. 

We discussed isotropic classes of functions with mixed smoothness. Isotropic means that all variables play the same role in the definition of our smoothness classes. In the hyperbolic cross approximation theory anisotropic classes of functions with mixed smoothness are of interest and importance. We give the corresponding definitions. Let $\br=(r_1,\dots,r_d)$ be such that $0<r_1=r_2=\dots=r_\nu<r_{\nu+1}\le r_{\nu+2}\le\dots\le r_d$ with $1\le \nu \le d$. 
For $\bx=(x_1,\dots,x_d)$ denote
$$
F_\br(\bx) := \prod_{j=1}^d F_{r_j}(x_j)
$$
and
$$
\bW^\br_p := \{f:f=\varphi\ast F_\br,\quad \|\varphi\|_p \le 1\}.
$$
We now proceed to classes $\bH^\br_q$ and $\bB^\br_{q,\theta}$.
Define
$$
\|f\|_{\bH^\br_q}:= \sup_\bs \|\delta_\bs(f)\|_q 2^{(\br,\bs)},
$$
and for $1\le \theta <\infty$ define
$$
\|f\|_{\bB^\br_{q,\theta}}:= \left(\sum_{\bs}\left(\|\delta_\bs(f)\|_q 2^{(\br,\bs)}\right)^\theta\right)^{1/\theta}.
$$
We will write $\bB^\br_{q,\infty}:=\bH^\br_q$. 
Denote the corresponding unit ball
$$
\bB^\br_{q,\theta}:= \{f: \|f\|_{\bB^\br_{q,\theta}}\le 1\}.
$$
It is known that in many problems of estimating asymptotic characteristics the anisotropic classes of functions of $d$ variables with mixed smoothness behave in the same way as isotropic classes of functions of $\nu$ variables (see, for instance, \cite{Tmon}). It is clear that the above remark holds for the lower bounds. To prove it for the upper bounds one needs to develop, in some cases, a special technique. The techniques developed in this paper work for the anisotropic classes as well. For instance, the main Lemma \ref{L2.1} is replaced by the following lemma.

\begin{Lemma}\label{L6.1} Denote $r:=r_1$. Define for $f\in L_1$
$$
f_{l,\br}:=\sum_{\bs:rl\le (\br,\bs)<r(l+1)}\delta_\bs(f), \quad l\in \N_0.
$$
Consider the class
$$
\bW^{\br,a,b}_A:=\{f: \|f_{l,\br}\|_A \le 2^{-al}l^{(\nu-1)b}\}.
$$
Then for $2\le p<\infty$ and $a>0$ there is a constructive method based on greedy algorithms, which provides the bound
\begin{equation}\label{2.13''}
\sigma_m(\bW^{\br,a,b}_A)_p \ll  m^{-a-1/2} (\log m)^{(\nu-1)(a+b)}.      
\end{equation}
For $p=\infty$ we have
\begin{equation}\label{2.14''}
\sigma_m(\bW^{\br,a,b}_A)_\infty \ll  m^{-a-1/2} (\log m)^{(\nu-1)(a+b)+1/2}.      
\end{equation}

\end{Lemma}

\begin{Proposition}\label{P6.1} Results of Section 2 hold for the anisotropic classes of functions with mixed smoothness with $r=r_1$ and $d$ replaced by $\nu$. 
\end{Proposition}


\begin{thebibliography}{9999}

\bibitem{Be2} E. S. Belinskii, Approximation by a "floating" system of exponentials on classes of smooth periodic functions, Matem. Sb.  {\bf 132} (1987), 20-27; English translation in Math. USSR Sb. {\bf 60} (1988). 

\bibitem{Be1} E.S. Belinskii, Decomposition theorems and approximation by a "floating" system of exponentials, Transactions of the American Mathematical Society, {\bf 350} (1998), 43--53. 

\bibitem{DT1}  R.A. DeVore and V.N. Temlyakov, Nonlinear approximation by trigonometric sums, J. Fourier Analysis and Applications, {\bf 2}
(1995),  29--48.

\bibitem{DKTe}  S.J. Dilworth, D. Kutzarova, V.N. Temlyakov,  Convergence of some Greedy Algorithms in Banach spaces, \emph{ The J. Fourier Analysis and Applications} 
 \textbf{ 8} (2002), 489--505. 
 
 \bibitem{DU} Dinh Dung and T. Ullrich, Lower bounds for the integration error for multivariate functions with mixed smoothness and optimal Fibonacci cubature for functions on the square,
manuscript, 2014. 

\bibitem{DGDS} M. Donahue, L. Gurvits, C. Darken and E. Sontag, Rate of convex approximation in non-Hilbert spaces, Constructive Approx., {\bf 13} (1997), 187--220.

 \bibitem{G}   E.D. Gluskin,  Extremal properties of orthogonal parallelpipeds and their application to the geometry of Banach spaces, \emph{ Math USSR Sbornik} \textbf{ 64}  (1989),  85--96.

 \bibitem{I} R.S. Ismagilov, Widths of sets in normed linear spaces and the approximation of functions by trigonometric polynomials, Uspekhi Mat. Nauk, {\bf 29} (1974),  161--178;   English transl. in 
 Russian Math. Surveys, {\bf 29} (1974).

\bibitem{KTE1} B.S. Kashin and V.N. Temlyakov, On best $m$-term approximations and the entropy of sets in the space $L^1$, {\em Math. Notes}, {\bf 56} (1994), 1137--1157.

 \bibitem{M}   V.E. Maiorov,  Trigonometric diameters of the Sobolev classes $W^r_p$ in the space $L_q$, \emph{ Math. Notes} \textbf{ 40} (1986),  590--597. 
 
 \bibitem{Mk} Y. Makovoz, On trigonometric $n$-widths and their generalizations, J. Approx. Theory {\bf 41} (1984), 361-366. 
 
 \bibitem{Rom1} A.S. Romanyuk, Best $M$-term trigonometric approximations of Besov classes of periodic functions of several variables, Izvestia RAN, Ser. Mat. {\bf67} (2003), 61--100; English transl. in Izvestiya: Mathematics (2003), 67(2):265. 
 
 \bibitem{St} S. B. Stechkin, On the best approximation of given classes of functions by arbitrary polynomials,Uspekhi Mat. Nauk {\bf 9} (1954), 133-134 (in Russian).

\bibitem{T29} V.N. Temlyakov, Approximation of Periodic Functions of Several Variables by Bilinear Forms, Izvestiya AN SSSR, {\bf 50} (1986), 137--155; English transl. in Math. USSR Izvestija, {\bf 28} (1987), 133--150. 

\bibitem{T23} V.N. Temlyakov, Approximate recovery of periodic functions of several variables, Mat. Sb., {\bf 128} (1985), 256--268; English transl. in   Math. USSR-Sb {\bf 56} (1987), 249--261.

\bibitem{Tmon} V.N. Temlyakov, Approximation of functions with bounded mixed derivative, Trudy MIAN, {\bf 178} (1986), 1--112. English transl. in Proc. Steklov Inst. Math., {\bf 1} (1989).

\bibitem{TE2}  V.N. Temlyakov, Estimates of the asymptotic characteristics of classes of functions with bounded mixed derivative or difference, {\em Trudy Matem. Inst. Steklov}, {\bf 189}   (1989),    138--168;
  English transl. in {\em Proceedings of the Steklov Institute of Mathematics}, 1990,   Issue 4, 
161--197.  

\bibitem{Tem22}  V.N. Temlyakov,  On a way of obtaining lower estimates 
for the errors of quadrature formulas, Matem. Sbornik, {\bf 181} (1990),
 1403--1413;  English transl. in  Math.
USSR Sbornik, {\bf 71} 
(1992). 

\bibitem{VT51} V.N. Temlyakov, On Approximate Recovery of Functions with Bounded Mixed Derivative, J. Complexity, {\bf 9} (1993), 41--59.

\bibitem{TE3}  V.N. Temlyakov, An inequality for trigonometric polynomials and its application for estimating the entropy numbers, {\em J. Complexity}, 
{\bf 11}   (1995),     293--307.

\bibitem{VT59} V.N. Temlyakov, An inequality for trigonometric polynomials and its application for estimating the Kolmogorov widths, East J. Approximations, {\bf 2} (1996), 253--262.

\bibitem{TBook}  V.N. Temlyakov, {\em Approximation of periodic functions},  Nova Science Publishes, Inc., New York.,  1993.

\bibitem{T1} V.N. Temlyakov, Nonlinear Kolmogorov's widths, Matem. Zametki, {\bf 63} (1998), 891--902.

   \bibitem{T13}  V.N. Temlyakov,  
  Weak greedy algorithms,
\emph{  Advances in Comput. Math.} \textbf{ 12}  (2000),  213--227.
  
 \bibitem{T8}  V.N. Temlyakov,  
  Greedy algorithms in Banach spaces,
\emph{  Advances in Comput. Math.} \textbf{ 14}  (2001),  277--292.

  
 \bibitem{T11} V.N. Temlyakov,   Cubature formulas and related questions,
\emph{  J. Complexity}   \textbf{ 19}  (2003),  352--391.

\bibitem{T12} V.N. Temlyakov, Greedy-Type Approximation in Banach Spaces and Applications, Constr. Approx. {\bf 21} (2005), 257--292.


\bibitem{Tbook} V.N. Temlyakov, Greedy approximation, Cambridge University Press, 2011.

\bibitem{Tappr} V.N. Temlyakov, An inequality for the entropy numbers and its application, 
J. Approx. Theory, {\bf 173} (2013), 110--121.

\bibitem{T144} V.N. Temlyakov, Sparse approximation and recovery by greedy algorithms in Banach spaces, Forum of Mathematics, Sigma, {\bf 2} (2014), e12, 26 pages;
 IMI Preprint, 2013:09, 1--27; arXiv:1303.6811v1, 27 Mar 2013.







  \end{thebibliography}
\end{document}